\documentclass{article}
\usepackage[utf8x]{inputenc}
\usepackage{amssymb}
\usepackage{amsmath}
\usepackage{mathrsfs}
\usepackage{bm}
\usepackage{amsthm}
\usepackage{enumitem}
\usepackage{color}
\usepackage{graphicx}
\usepackage{bm}
\usepackage{ulem}
\usepackage{upgreek}

\usepackage{accents}

\newcommand{\ci}[1]{\mathscr{#1}}
\newcommand{\g}[1]{\mathfrak{#1}}

\newcommand{\alfa}{\alpha}
\newcommand{\R}{\mathbf{R}}
\newcommand{\C}{\mathbf{C}}
\newcommand{\de}{\partial}
\renewcommand{\H}{\mathbf{H}}
\newcommand{\N}[1]{\left\lVert#1\right\rVert}

\newcommand{\bra}{\left\langle}
\newcommand{\ket}{\right\rangle}

\renewcommand{\phi}{\varphi}
\newcommand{\con}[1]{\overline{#1}}

\newcommand{\cerchio}[1]{\accentset{\smash{\raisebox{-0.12ex}{$\scriptstyle\circ$}}}{#1}\rule{0pt}{2.3ex}}

\newtheorem{proposizione}{Proposition}[section]
\newtheorem{teorema}[proposizione]{Theorem}
\newtheorem{lemma}[proposizione]{Lemma}

\theoremstyle{definition}

\theoremstyle{remark}

\title{A CR structure with blowing up solutions to the CR Yamabe problem}
\author{Claudio Afeltra, Andrea Pinamonti}
\date{}

\begin{document}

\maketitle

\begin{abstract}
 We prove the existence of a CR structure on $S^3$ such that the set of solutions to the CR Yamabe problem is not compact and admits a blowing-up sequence.
 Such CR structure is built deforming the standard CR structure of $S^3$ in the direction of the Rossi sphere CR structure on small balls, and the existence of the blowing-up sequence of solutions is proved through the Lyapunov-Schmidt method.
\end{abstract}

\section{Introduction}
The efforts to solve the Yamabe problem, that is the problem of finding a Riemannian metric conformally equivalent to a given one and with constant scalar curvature, have lead to huge developments in the field of geometric analysis, until the final part of the solution was completed by Richard Schoen in 1984, after previous contributions by Yamabe, Trudinger, Aubin and others (see \cite{LP} for a survey of the history and proof of the problem).

After the solution of the problem, it was conjectured that the set of the solutions thereto is compact. After numerous partial results, the conjecture was proved to be true in dimension $n\le 24$ by Khuri, Marques and Schoen (see \cite{KMS}), while surprisingly in higher dimension it was proved to be false by Brendle (for $n\ge 52$, see \cite{B}) and Brendle and Marques (for $25\le n\le 51$, see \cite{BM}).

\vspace{3mm}

The same problems can be posed in CR geometry.
CR geometry is the study of manifolds endowed with a certain geometric structure which arises in the study of real hypersurfaces of complex manifolds (see Section \ref{SezionePreliminari} for definitions). These manifolds carry a natural cotact structure, and the choice of a contact form determines a rich geometric structure, including a connection known as Tanaka-Webster connection, and a Riemannian metric.
Analogously to Riemannian geometry, the curvature tensor of the connection can be contracted twice with the metric giving rise to a scalar invariant known as Webster curvature.
Since the choice of a contact form is determined up to the multiplication by a nowhere zero function, it is natural to pose the so called CR Yamabe problem, that is the problem of finding a contact form such that the associated Webster curvature is constant.

This problem has striking similarities with the (Riemannian) Yamabe problem.
In particular the non conformally flat, high dimensional case was solved by Jerison and Lee in 
\cite{JL1} and \cite{JL3} by following the same strategy of the analogous Riemannian case; the remaining cases were solved by Gamara and Yacoub in \cite{G} and \cite{GY} using the method of critical points at infinity.

\vspace{3mm}

It is therefore natural to study the question of the compactness of the set of solutions to the CR Yamabe problem.
The study of this problem in Riemannian geometry uses blow-up analysis, which consists in taking a sequence which supposedly violates compactness and rescaling and dilating it around a local maximum in such a way to get a solution to the Yamabe problem in $\R^n$. The analysis therefrom  proceeds using the classification theorem by Caffarelly, Gidas and Spruck (see \cite{CGS}).
Unfortunately the reproduction of this proof in CR geometry presents the obstacle that blowing up a CR manifold around a point one gets the Heisenberg group $\H^n$, and for it there does not currently exist an analogous to the theorem of Caffarelly, Gidas and Spruck for $\H^n$.

In 2023, Catino, Li, Monticelli and Roncoroni proved such a classification theorem for $\H^1$ (see \cite{CLMR}). This allowed the first author of this work to prove a compactness theorem (see \cite{A}).
With respect to the Riemannian case, such theorem had an additional hypothesis, that is that the pseudohermitian mass of the blow-up structure at every point is positive. The pseudohermitian mass is a quantity defined for asymptotically flat CR manifolds which is analogous to the ADM mass for asymptotically flat manifolds, and similarly appears in the expansion of the conformal sublaplacian; but unlike its Riemannian counterpart, it can be negative also for CR manifolds with positive Webster curvature, as was showed by an example by Cheng, Malchiodi and Yang in \cite{CMY}.

In \cite{A} the author raised the question of whether such hypothesis on the pseudohermitian mass was necessary in the compactness theorem proved therein.
In this article, we answer affermatively to the question and find a counterexample.

\begin{teorema}\label{Teorema}
 There exists a CR structure on $S^3$, not equivalent to the standard one, such that the associated CR Yamabe equation
 $$L_{J,\theta}u = 2u^3$$
 (where $L_{J,\theta}$ is the conformal sublaplacian, defined in Section \ref{SezionePreliminari}) has a set of solutions $\{u_k\}_{k\in\mathbf{N}}$ with $\max u_k\to\infty$.
\end{teorema}

In order to build this counterexample, we took inspiration by the strategy of Brendle in \cite{B} for building an analogous counterexample for Riemannian manifolds of dimension $n\ge 52$.
The reason why the compactness theorem holds only in dimension $n\le 24$ is that a certain quadratic form arising from the blow-up analysis is positive definite in such dimensions, while this is not true for $n\ge 25$.
The starting point of Brendle's proof was to deform the Euclidean metric of $\R^n$ in a direction in which the quadratic form is negative.
Then he considered the manifold consisting of ``bubbles'' (that is, solutions to the Yamabe problem for the Euclidean metric) and applied the Lyapunov-Schmidt method, a method to find critical points of a perturbed functional starting from a manifold of critical points of the unperturbed functional.
Through delicate estimates, he proved that the existence of deformed metrics admitting arbitrarily large solutions to the Yamabe problem, and finally he build the required counterexample by modifying the Euclidean metric on the union of a sequence of disjoint balls $B_{r_k}(x_k)$.

In our case, since the essential reason for the failure of the compactness theorem for three-dimensional CR manifolds is the fact that the pseudohermitian mass can be negative, our strategy is to deform the standard CR structure of $\H^1$ in the direction of the only known counterexample to the CR positive mass theorem, which are known as Rossi spheres $S_s^3$.
In \cite{CMY} it is proved that the pseudohermitian mass of their blow-up structure satisfies $m_s = -18\pi s^2 + o(s^2)$, and fine estimates for the associated CR Yamabe functional.
By making use of these results, we are able to study the perturbed functional restricted to an appropriate region of the manifold of bubbles and to build a CR structure with a blowing up sequence of solutions to the CR Yamabe problem.


\section{Preliminaries}\label{SezionePreliminari}
In this section we outline the basic definitions in CR geometry that we will need.
For an introduction to CR geometry we refer to \cite{DT}.

A CR structure on a three-dimensional manifold $M$ is a subbundle $\ci{H}$ of $TM\otimes\C$ such that $\ci{H}\cap\overline{\ci{H}}=0$.
The Levi distribution is $H(M)=\g{Re}(\ci{H}\oplus\overline{\ci{H}})$, and it carries the natural complex structure defined by $J(Z+\con{Z})=i(Z-\con{Z})$.
If $H(M)$ is fixed, a CR structure can be equivalently defined as an complex structure $J$ on it. This is the approach we will use to build the desired counterexample.

On an orientable CR manifold a contact form $\theta$ such that $\left.\theta\right|_{H(M)}=0$ can be chosen.
Since in this article such a form will be considered fixed and the main concern is the CR structure, we will generally omit the dependence of the various geometric objects by $\theta$.

The choice of a contact form on a CR manifold provides a rich geometry, in particular it allows to define a subriemannian metric $G_{\theta}$ on $H(M)$ and a connection $\nabla$ called the Tanaka-Webster connection.
Contracting twice the curvature tensor of such connection with $G_{\theta}$ allows to define a scalar invariant known as Webster curvature which is analogous to the scalar curvature in Riemannian geometry.

The contact form gives rise to the volume form $\theta\wedge d\theta$, which in this article will be mostly omitted in the integrals. Composing the divergence with respect to such volume form with the subriemannian gradient a second order subelliptic differential operator $\Delta_J$ called the sublaplacian is defined.

Any other contact form for a given CR structure is of the form $u^2\theta$ with $u\ne 0$ everywhere.
The Webster curvature with respect to the contact form $\widetilde{\theta}=u^2\theta$ is given by the formula
$$\widetilde{R} = u^{-3}(-4\Delta_J+R)u$$
and therefore, calling $L_J= -4\Delta_J+R$ the conformal sublaplacian, the equation for the prescribed Webster curvature is
\begin{equation}\label{EquazioneYamabeCR}
 L_Ju=2u^3.
\end{equation}

The most important three-dimensional CR structure, playing the same role as $\R^n$ in the category of Riemannian manifolds, is the Heisenberg group $\H^1$.

It is defined as $\C\times\R$ with the group law $(z,t)\cdot(w,s) = (z+w,t+s+2\g{Im}(z\cdot\con{w}))$, and it carries the left invariant CR structure generated by the left-invariant vector field
$$Z= \frac{\de}{\de z} +i\overline{z}\frac{\de}{\de t}$$
and the contact structure
$$\theta= dt+i\left( zd\con{z}-\con{z}dz\right).$$
We call $J_0$ the associated complex structure on $H(\H^1)$.
$\H^1$ admits a family of CR and group automorphisms analogous to the dilations of $\R^n$, defined by $\delta_{\lambda}(z,t)=(\lambda z,\lambda^2 t)$, and called Heisenberg dilations.
The Korányi norm is defined as $|(z,t)|=\left(|z|^4+t^2\right)^{\frac{1}{4}}$, and satisfies $|\delta_{\lambda}x|=\lambda|x|$.

The solutions to the CR Yamabe equation \eqref{EquazioneYamabeCR} on $\H^1$ are
$$U(z,t)=\frac{c_1}{\left(t^2+(1+|z|^2)^2\right)^{\frac{1}{2}}}$$
and the translations and dilations thereof, $u_{x,\lambda}=\lambda U\circ\delta_{\lambda}\circ L_x$ (where $L_x(y)=x^{-1}y$).
We will use also the vector field $T=\frac{\de}{\de t}$ (the Reeb vector field) and the generator of the dilations
$$\Xi = \left(zZ+\con{z}Z\right) +2tT.$$
Let $X$ be the Hilbert space
$$X= \{u\in L^2_{\rm{loc}}(\H^1)| \N{\nabla_{J_0} u}_{L^2} <\infty\}$$
with the scalar product
$$\bra u,v\ket_X= \int_{\H^1}\nabla_{J_0}u\cdot\nabla_{J_0}v.$$
By standard functional analysis, $\Delta_{J_0}$ is an isomorphism between $X$ and the dual thereof $X'$.

By the Sobolev inequality in the Heisenberg group
$\N{u}_{L^4}\le C \N{u}_X$.

We also define the Heisenberg analogous of the $C^2$ norm,
$$\N{u}_{\Gamma^2} = \N{u}_{L^{\infty}}+\N{Zu}_{L^{\infty}}+\N{\con{Z}u}_{L^{\infty}}+\N{Z^2u}_{L^{\infty}}+\N{\con{Z}^2u}_{L^{\infty}} + \N{Z\con{Z}u}_{L^{\infty}}+\N{\con{Z}Zu}_{L^{\infty}}.$$


\subsection{Rossi spheres}\label{SottosezioneRossi}
The standard CR structure on $S^2$ induced by the immersion in $\C^2$ is generated by
$$W= \overline{w_2}\frac{\de}{\de w_1} - \overline{w_1}\frac{\de}{\de w_2}.$$
Then the Rossi sphere are defined as $S^2$ with the contact form generated by $W+s\con{W}$.

We recall that the Kelvin transform is defined as the function $F:S^3\setminus(0,-1)\to\H^1$ given by
$$F(w_1,w_2)=\left(\frac{w_1}{1+w_2},\g{Re}\left(i\frac{1-w_2}{1+w_2}\right)\right)$$
whose inverse is
$$F^{-1}(z,t) = \left(\frac{2iz}{t+i(1+|z|^2)},\frac{-t+i(1-|z|^2)}{t+i(1+|z|^2)}\right).$$
It is known that the push-forward through $F$ of the standard CR structure on $S^2$ is the Heisenberg CR structure. We want to compute the push-forward of the Rossi CR structure.

We have
$$dF\left(\frac{\de}{\de w_1}\right) = \frac{1}{1+w_2}\frac{\de}{\de z} = \frac{1}{1+\frac{-t+i(1-|z|^2)}{t+i(1+|z|^2)}}\frac{\de}{\de z} = \frac{t+i(1+|z|^2)}{2i}\frac{\de}{\de z} $$
and
$$dF\left(\frac{\de}{\de w_2}\right) = -w_1\frac{1}{(1+w_2)^2}\frac{\de}{\de z} +\frac{1}{2}i\frac{-(1+w_2) -(1-w_2)}{(1+w_2)^2}\frac{\de}{\de t}= -w_1\frac{1}{(1+w_2)^2}\frac{\de}{\de z} -i\frac{1}{(1+w_2)^2}\frac{\de}{\de t}=$$
$$= -\frac{2iz}{t+i(1+|z|^2)}\frac{(t+i(1+|z|^2))^2}{-4}\frac{\de}{\de z} -i\frac{(t+i(1+|z|^2))^2}{-4}\frac{\de}{\de t}=$$
$$= \frac{iz}{2}(t+i(1+|z|^2))\frac{\de}{\de z} +\frac{i}{4}(t+i(1+|z|^2))^2\frac{\de}{\de t}$$
and so
$$dF(W)= \frac{-t-i(1-|z|^2)}{t-i(1+|z|^2)}dF\left(\frac{\de}{\de w_1}\right) - \frac{-2i\overline{z}}{t-i(1+|z|^2)}dF\left(\frac{\de}{\de w_2}\right) =$$
$$= \frac{-t-i(1-|z|^2)}{t-i(1+|z|^2)}\frac{t+i(1+|z|^2)}{2i}\frac{\de}{\de z} + \frac{2i\overline{z}}{t-i(1+|z|^2)}\left(\frac{iz}{2}(t+i(1+|z|^2))\frac{\de}{\de z} +\frac{i}{4}(t+i(1+|z|^2))^2\frac{\de}{\de t}\right) = $$
$$= \frac{t+i(1+|z|^2)}{t-i(1+|z|^2)}\frac{i}{2}\left(t+i(1-|z|^2) +2i|z|^2\right)\frac{\de}{\de z} -\frac{1}{2}\overline{z}\frac{(t+i(1+|z|^2))^2}{t-i(1+|z|^2)}\frac{\de}{\de t}=$$
$$= \frac{(t+i(1+|z|^2))^2}{t-i(1+|z|^2)}\frac{i}{2}\frac{\de}{\de z} -\frac{1}{2}\overline{z}\frac{(t+i(1+|z|^2))^2}{t-i(1+|z|^2)}\frac{\de}{\de t}=$$
$$= \frac{i}{2}\frac{(t+i(1+|z|^2))^2}{t-i(1+|z|^2)}Z = \frac{i}{2}\frac{(t+i(1+|z|^2))^3}{t^2+(1+|z|^2)^2}Z.$$
Therefore the push-forward of the Rossi sphere CR structure through $F$ is generated by the vector field
$$Z^s = Z^s_1 = Z + s\frac{(t-i(1+|z|^2))^3}{(t+i(1+|z|^2))^3}\overline{Z} = Z + s\phi(z,t)\overline{Z}$$
where
$$ \phi(z,t) = \left(\frac{t-i(1+|z|^2)}{t+i(1+|z|^2)}\right)^3.$$

\section{Lyapunov-Schmidt method}
Let us define the four-dimensional submanifold of $X$
$$\ci{M} = \left\{ U_{x,\lambda} \;\middle|\; x\in\H^1, \lambda\in(0,\infty)\right\}$$
and $\ci{E}_{x,\lambda}$ the orthogonal to the tangent thereof at the point $U_{x,\lambda}$, that is
$$\ci{E}_{x,\lambda} = \operatorname{span}(Z_1U_{x,\lambda}, Z_{\con{1}}U_{x,\lambda}, TU_{x,\lambda}, \Xi U_{x,\lambda})^{\perp} =$$
$$= \left\{u\in X \;\middle|\; \int_{\H^1}\nabla_{J_0}u\nabla_{J_0}(Z_1U_{x,\lambda})=\int_{\H^1}\nabla_{J_0}u\nabla_{J_0}(Z_{\con{1}}U_{x,\lambda})=\right.$$
$$\left.=\int_{\H^1}\nabla_{J_0}u\nabla_{J_0}(TU_{x,\lambda})=\int_{\H^1}\nabla_{J_0}u\nabla_{J_0}(\Xi U_{x,\lambda})=0\right\}$$
We notice that the CR Yamabe equation for $J$ is the Euler-Lagrange equation for the functional
$$\ci{J}_J(u)= \int_{\H^1}uL_Ju - \int_{\H^1}u^4$$
in the space $X$.

\begin{proposizione}\label{DifferenzialeSecondoInvertibile}
 There exists a constant $\alfa$ such that if $J$ is a CR structure on $\H^1$ coinciding with $J_0$ on $\H^1\setminus B_1(0)$ and such that $\N{J-J_0}_{\Gamma^2}\le\alfa$ then for every $(x,\lambda)\in\H^1\times(0,\infty)$ the operator $d^2\ci{J}_J(U_{x,\lambda})$ is invertible on $\ci{E}_{x,\lambda}$, and the norms of $d^2\ci{J}_J(U_{x,\lambda})$ and $\left(\left.d^2\ci{J}_J(U_{x,\lambda})\right|_{\ci{E}_{x,\lambda}}\right)^{-1}$ are bounded independently by $x$, $\lambda$ and $J$.
\end{proposizione}

The proof of the above Proposition is standard in the Riemannian case (see \cite{B}); in this case the proof is the same except that Lemma 5 from \cite{MU} has to be used instead of the analogous result for $\R^n$, and the estimates in Propositions \ref{StimaCurvatura} and \ref{DifferenzaSublaplaciani}.

\begin{proposizione}\label{EsistenzaLS}
 There exists a constant $\alfa$ such that if $J$ is a CR structure on $\H^1$ coinciding with $J_0$ on $\H^1\setminus B_1(0)$ and such that $\N{J-J_0}_{\Gamma^2}\le\alfa$ then for every $(x,\lambda)\in\H^1\times(0,\infty)$ there exists an unique $v_{x.\lambda}\in\ci{E}_{x,\lambda}$ with $\N{v}\lesssim\alfa$ such that
 $$\pi_{\ci{E}_{x,\lambda}}\left(\nabla\ci{J}_J(U_{x,\lambda}+v_{x,\lambda})\right) = 0.$$
\end{proposizione}

\begin{proof}
 The equation to be solved can be written explicitely as
 $$\pi_{\ci{E}_{x,\lambda}}\circ\Delta_{J_0}^{-1}\left(L_J(U_{x,\lambda}+v_{x,\lambda})-2(U_{x,\lambda}+v_{x,\lambda})^3\right)=0,$$
 and therefore the solutions thereof are fixed points of the function
 $$F(v)= (\pi_{\ci{E}_{x,\lambda}}\circ\Delta_{J_0}^{-1}\circ(L_J-6U_{x,\lambda}^2))^{-1}\pi_{\ci{E}_{x,\lambda}}\circ\Delta_{J_0}^{-1}\left(-6U_{x,\lambda}v^2 - 2v^3 + L_JU_{x,\lambda} -2U_{x,\lambda}^3\right).$$
 Using Proposition \ref{DifferenzialeSecondoInvertibile}, the fact that the adjoint of the immersion $X\to L^4$ is $L^{\frac{4}{3}}\to X'$, and Propositions \ref{StimaCurvatura} and \ref{DifferenzaSublaplaciani}, for any $v\in\ci{E}_{x,\lambda}$
 $$\N{F(v)}_X \lesssim \N{-6U_{x,\lambda}v^2 - 2v^3 + L_JU_{x,\lambda} -2U_{x,\lambda}^3}_{X'} \lesssim$$
 $$\lesssim \N{-6U_{x,\lambda}v^2 - 2v^3 + L_JU_{x,\lambda} -2U_{x,\lambda}^3}_{L^{\frac{4}{3}}} \lesssim \N{U_{x,\lambda}v^2}_{L^{\frac{4}{3}}} + \N{v^3}_{L^{\frac{4}{3}}} + \N{L_JU_{x,\lambda} -2U_{x,\lambda}^3}_{L^{\frac{4}{3}}} \lesssim$$
 $$\lesssim \N{U_{x,\lambda}}_{L^4}\N{v^2}_{L^2} + \N{v^3}_{L^{\frac{4}{3}}} + \N{L_JU_{x,\lambda} -L_{J_0}U_{x,\lambda}}_{L^{\frac{4}{3}}} \lesssim$$
 $$\lesssim \N{v}_{L^4}^2 + \N{v}_{L^4}^3 + \N{(\Delta_J-\Delta_{J_0})U_{x,\lambda}}_{L^{\frac{4}{3}}} + \N{(R_J-R_{J_0})U_{x,\lambda}}_{L^{\frac{4}{3}}} \lesssim$$
 $$\lesssim \N{v}_X^2 + \N{v}_X^3 + \alfa\left(\int_{B_1(0)}\left(\lambda\lambda^2(1+\lambda|x|)^{-4}\right)^{\frac{4}{3}} + \int_{B_1(0)}\left(\lambda\lambda(1+\lambda|x|)^{-3}\right)^{\frac{4}{3}}\right)^{\frac{3}{4}} +$$
 $$ + \N{(R_J-R_{J_0})}_{L^2}\N{U_{x,\lambda}}_{L^4} \lesssim$$
 $$\lesssim \N{v}_X^2 + \N{v}_X^3 + \alfa\lambda^3\left(\lambda^{-4}\int_{B_{\lambda}(0)}\left(1+|x|\right)^{-\frac{8}{3}}\right)^{\frac{3}{4}} + \alfa\lambda^2\left(\lambda^{-4}\int_{B_{\lambda}(0)}(1+|x|)^{-4}\right)^{\frac{3}{4}} + \alfa \lesssim$$
 $$\lesssim \N{v}_X^2 + \N{v}_X^3 + \alfa.$$
 Therefore a ball of appropriate radius $\lesssim\alfa$ is invariant for $F$.
 
 On such a ball, with similar computations,
 $$\N{F(v)-F(w)}_X \lesssim  \N{U_{x,\lambda}(v^2-w^2)}_{X'} + \N{(v^3-w^3)}_{X'} \lesssim$$
 $$\lesssim \N{(v^2-w^2)}_{L^2} + \N{(v^3-w^3)}_{L^{\frac{4}{3}}}\lesssim$$
 $$\lesssim \N{v+w}_{L^4}\N{v-w}_{L^4} + \N{v^2+vw+w^2}_{L^2}\N{v-w}_{L^4}\lesssim$$
 $$\lesssim \alfa\N{v-w}_X +\alfa^2\N{v-w}_X$$
 therefore $F$ is a contraction thereon up to choosing $\alfa$ small enough.
\end{proof}

Finally we state a Proposition according to which restricted critical points to the ``perturbed'' manifold which was got in the above Proposition are unrestricted critical points, allowing us to reduce our study to a functional on the four-dimensional submanifold $\ci{M}$ of $X$.
The proof thereof is a classical tool in the Lyapunov-Schmidt method (see \cite{AM}).

\begin{proposizione}\label{ProposizioneLS}
 If $(x,\lambda)\in\H^1\times(0,\infty)$ is a critical point of $\ci{J}_J(U_{x,\lambda}+v_{x,\lambda})$, then $\nabla\ci{J}_J(U_{x,\lambda}+v_{x,\lambda}) = 0$.
\end{proposizione}

\section{Construction of CR structures with large solutions}
Let $x_n=(\frac{1}{n},0,0)\in\H^1$ and let $r_k,R_k,s_k$ sequences converging to zero and such that the balls $B_{R_k}(x_k)$ are disjoint. Let $J$ be a CR structure coinciding with $J_0$ on $\H^1\setminus\bigcup B_{Ar_k}(x_k)$, with $J_{s_k}$
on $B_{r_k}(x_k)$, and on $B_{Ar_k}\setminus B_{r_k}$ with $J_f$, with $|f|\le s_k$.

Let us define
$$\Omega_k = \left\{U_{x,\lambda}\;\middle|\; |x-x_k|<R_k, \frac{\alfa}{R_k}<\lambda<\frac{\beta}{r_k}\right\}\subset\ci{M}$$
with $\alfa$, $\beta$ that will be chosen later.

Then we want to prove that, up to choosing the parameters appropriately, for every $k$ there exists a critical point of $\ci{J}_J(U_{x,\lambda}+v_{x,\lambda})$ in $\Omega_k$, that is, a solution to the CR Yamabe problem for $J$ that is approximately a bubble centered at $x_k$. Through this, we will prove that $J$ is the required counterexample.

In order to carry out our estimates, we will need to use approximated solutions to the CR Yamabe equation for the Rossi spheres more precise than standard balls, which were defined in \cite{CMY}.
Let $\breve{U}_{x,\lambda,s}$ be these approximated solutions carried to $\H^1$ through the Cayley transform $F$ defined in Subsection \ref{SottosezioneRossi}.
We will omit the dependence by $s$ when unnecessary.
They satisfy the estimates
$$\breve{U}_{x,\lambda,s}(x) \lesssim \frac{\lambda}{(1+\lambda|x|)^2}$$
and for $|x|$ small
$$L_{J_s}\breve{U}_{x,\lambda,s} = \breve{U}_{x,\lambda,s}^3 \left( 2 + O(|x|^3) +  \lambda^{-2} O(|x|^2)  \right)+$$
\begin{equation}\label{StimaBollaApprossimata}
 +  \breve{U}_{x,\lambda,s}^5 \left[ |x|^4 O(4 + \lambda^2 |x|^2)s^2  + O(|x|^5) + O(\lambda^2 |x|^7) \right].
\end{equation}
Furthermore on these approximated solutions, the functional $\ci{J}_{J_s}$ can be estimated precisely on them.

\begin{lemma}\label{EspansioneCMY}
 $$\ci{J}_{J_s}(\breve{U}_{x,\lambda,s}) = 4\pi^2 + 24\pi\frac{s^2}{\lambda^2}(1+o_s(1))+ O\left(\frac{1}{\lambda^3}\right).$$
\end{lemma}
This expansion follows from the proof of Lemma 7.3 in \cite{CMY} (the authors there study the CR Yamabe functional and not $\ci{J}_{J_s}$, but they estimate the numerator and the denominator substantially independently).

Let us fix $k$ and consider the approximated solutions $\breve{U}_{x,\lambda,s_k}$ for $x\in B_{r_k}(x_k)$ and $\lambda\gtrsim\frac{1}{r_k}$.
We can change the definition of $\breve{U}_{x,\lambda,s_k}$ so that they coincide with $U_{x,\lambda}$ on $\H^1\setminus B_{Ar_k}(x_k)$, because in the proofs of the results of \cite{CMY}, only the local behavior of the approximated solutions is used, except at the end of the proof of Corollary B.2, when the asymptotics at infinity are also used, but our modification does not change that.
In this way we can also assume that
\begin{equation}\label{Simmetria1}
 \breve{U}_{x,\lambda,0}\equiv U_{x,\lambda}
\end{equation}
and that, using the operator $\iota$ defined in formula (23) in \cite{CMY},
\begin{equation}\label{Simmetria2}
 \iota^*\breve{U}_{x,\lambda,s_k} = \breve{U}_{x,\lambda,-s_k}
\end{equation}
In such a way, reasoning as in Subsection B.3 in \cite{CMY}, the estimate in Lemma \ref{EspansioneCMY} can be improved to
\begin{equation}\label{EspansioneCMY2}
 \ci{J}_{J_s}(\breve{U}_{x,\lambda,s}) = 4\pi^2 + 24\pi\frac{s^2}{\lambda^2}+ O\left(\frac{s^2}{\lambda^3}\right).
\end{equation}

\begin{lemma}\label{StimaDifferenziale}
 For $\lambda$ large it holds that $\N{d\ci{J}_{J_{s_k}}(\breve{U}_{x,\lambda,s_k})}\lesssim\frac{s_k^2}{\lambda^2}$, and if furthermore $\lambda\ge\frac{C}{R_k}$, with $C$ big enough $x\in B_{r_k}(x_k)$ then $\N{d\ci{J}_J(\breve{U}_{x,\lambda,s_k})}\lesssim\frac{s_k}{\lambda^2}$.
\end{lemma}
The proof of the above Proposition is a variant of the proof of Corollary B.2 in \cite{CMY}.

Thanks to the above Proposition, we can perform the contraction argument of Proposition \ref{EsistenzaLS} starting from $\breve{U}_{x,\lambda,s_k}$ instead of $U_{x,\lambda}$, in a ball of radius $\frac{Cs_k^2}{\lambda^2}$.

\begin{lemma}\label{LSapprox}
 For every $x\in B_{r_k}(x_k)$ and $\lambda\ge\frac{C}{R_k}$ there exists $\breve{v}_{x,\lambda,s_k}\in\ci{E}_{x,\lambda}$ with $\N{\breve{v}_{x,\lambda,s_k}}_X\lesssim\frac{s_k^2}{\lambda^2}$ such that
 $$\pi_{\ci{E}_{x,\lambda}}\left(\nabla\ci{J}_J(\breve{U}_{x,\lambda,s_k}+\breve{v}_{x,\lambda,s_k})\right) = 0.$$
\end{lemma}

Notice that, thanks to the uniqueness statement in the contraction theorem, $U_{x,\lambda}+v_{x,\lambda,s_k}= \breve{U}_{x,\lambda,s_k}+\breve{v}_{x,\lambda,s_k}$.

\begin{lemma}\label{StimaInterna}
 $$\ci{J}_{J}(U_{x,\lambda}+v_{x,\lambda,s_k}) = 4\pi^2 + 24\pi\frac{s^2}{\lambda^2}+ O\left(\frac{s_k^2}{\lambda^3}\right).$$
\end{lemma}

\begin{proof}
 Thanks to Lemmas \ref{StimaDifferenziale} and \ref{LSapprox}
 $$\ci{J}_{J}(U_{x,\lambda}+v_{x,\lambda,s_k}) - \ci{J}_{J_{s_k}}(\breve{U}_{x,\lambda,s_k}) = \ci{J}_{J}(\breve{U}_{x,\lambda,s_k}+\breve{v}_{x,\lambda,s_k}) -\ci{J}_{J_{s_k}}(\breve{U}_{x,\lambda,s_k}) =$$
 $$=\ci{J}_{J}(\breve{U}_{x,\lambda,s_k}) - \ci{J}_{J_{s_k}}(\breve{U}_{x,\lambda,s_k}) +O\left(\frac{s_k^4}{\lambda^6}\right) =$$
 $$= \int_{\H^1\setminus B_{r_k}(x_k)}\breve{U}_{x,\lambda,s_k}L_{s_k}\breve{U}_{x,\lambda,s_k}-\breve{U}_{x,\lambda,s_k}^4 +O\left(\frac{s_k^2}{\lambda^4}\right) = O\left(\frac{s_k^2}{\lambda^4}\right).$$
 The thesis follows thanks to \eqref{EspansioneCMY2}.
\end{proof}

\begin{lemma}\label{StimaFondo}
 If $|x-x_k|\le R_k$ and $\lambda = \frac{\alfa}{R_k}$, then
 $$\ci{J}_J(U_{x,\lambda}+v_{x,\lambda,s_k}) = 4\pi^2 + O\left(\frac{s_k^2r_k^4}{R_k} + \frac{s_k^2r_k^6}{R_k^2}+s_k^4R_k^6\right).$$
\end{lemma}

\begin{proof}
 Thanks to \eqref{Simmetria1} and \eqref{Simmetria2}, the estimate \eqref{StimaBollaApprossimata} can be improved to
 $$L_{J_s}\breve{U}_{x,\lambda,s_k} - 2\breve{U}_{x,\lambda,s_k}^3 =$$
\begin{equation}
 = \breve{U}_{x,\lambda,s_k}^3 \left( s_k^2O(|x|^3) +  s_k^2\lambda^{-2} O(|x|^2)  \right)
+  \breve{U}_{x,\lambda,s_k}^5 \left[ |x|^4 O(4 + \lambda^2 |x|^2)s_k^2  + O(s_k^2|x|^5) + O(s_k^2\lambda^2 |x|^7) \right].
\end{equation}
 Therefore 
 $$\ci{J}_J(U_{x,\lambda}+v_{x,\lambda,s_k}) = \ci{J}_{J}(\breve{U}_{x,\lambda,s_k}+\breve{v}_{x,\lambda,s_k}) = \ci{J}_{J}(\breve{U}_{x,\lambda,s_k}) +O\left(\frac{s_k^4}{\lambda^6}\right) =$$
 $$= \int_{\H^1}U_{x,\lambda}^4 + \int_{B_{Ar_k}(x_k)}\breve{U}_{x,\lambda,s_k}L_J\breve{U}_{x,\lambda,s_k}-2\breve{U}_{x,\lambda,s_k}^4 +O\left(\frac{s_k^4}{\lambda^6}\right) =$$
 $$= 4\pi^2 + s_k^2\int_{B_{Ar_k}(x_k)}\left(\breve{U}_{x,\lambda,s_k}^4 \left( O(|x|^3) +  \lambda^{-2} O(|x|^2)  \right)+  \breve{U}_{x,\lambda,s_k}^6 \left[ |x|^4 O(4 + \lambda^2 |x|^2)  + O(|x|^5) + O(\lambda^2 |x|^7) \right]\right) +O\left(\frac{s_k^4}{\lambda^6}\right) =$$
 $$= 4\pi^2 + s_k^2\left(\frac{\alfa}{R_k}\right)^4O\left(\int_{B_{\alfa Ar_k/R_k}(x_k)}\left(\frac{R_k}{\alfa}\right)^4\left[\left(\frac{R_k}{\alfa}\right)^3|x|^3 + \lambda^{-2}\left(\frac{R_k}{\alfa}\right)^2|x|^2\right]\right) +$$
 $$+ s_k^2\left(\frac{\alfa}{R_k}\right)^4O\left(\int_{B_{\alfa Ar_k/R_k}(x_k)}\left(\frac{R_k}{\alfa}\right)^6\left[\left(\frac{R_k}{\alfa}\right)^4|x|^4 + \left(\frac{R_k}{\alfa}\right)^5|x|^5\right]\right) +O\left(\frac{s_k^4}{\lambda^6}\right) =$$
 $$= 4\pi^2 + s_k^2O\left(\frac{r_k^4}{R_k} + \frac{r_k^6}{R_k^2}\right) +O\left(s_k^4R_k^6\right).$$
\end{proof}

\begin{lemma}\label{StimaLaterale}
 If $|x-x_k|=R_k$ and $\frac{\alfa}{R_k}\le\lambda\le\frac{\beta}{r_k}$ then
 $$\ci{J}_J(U_{x,\lambda}+v_{x,\lambda,s_k}) = 4\pi^2 +O\left(\lambda^{-4}r_k^{-1}s_k^2 +\lambda^{-6}r_k^{-4}s_k^2\right) +O\left(\frac{s_k^4}{\lambda^6}\right).$$
\end{lemma}

\begin{proof}
 Proceeding as in the proof of Lemma \ref{StimaFondo},
 $$\ci{J}_J(U_{x,\lambda}+v_{x,\lambda,s_k}) = \ci{J}_{J}(\breve{U}_{x,\lambda,s_k}+\breve{v}_{x,\lambda,s_k}) = \ci{J}_{J}(\breve{U}_{x,\lambda,s_k}) +O\left(\frac{s_k^4}{\lambda^6}\right) =$$
 $$= \int_{\H^1}U_{x,\lambda}^4 + \int_{B_{Ar_k}(x_k)}\breve{U}_{x,\lambda,s_k}L_J\breve{U}_{x,\lambda,s_k}-2\breve{U}_{x,\lambda,s_k}^4 +O\left(\frac{s_k^4}{\lambda^6}\right) =$$
 $$= 4\pi^2 +O\left(s_k^2\lambda^{-4}\int_{x+\delta_{\lambda}(B_{Ar_k}(x_k)-x)}\left[\lambda^4|x|^{-8}(\lambda^{-3}|x|^3 + \lambda^{-4}|x|^{2})+\right.\right.$$
 $$\left.\left. + \lambda^6|x|^{-12}(\lambda^{-4}|x|^4(1+|x|^2)+ \lambda^{-5}|x|^5 + \lambda^{-5}|x|^7) \right]\right) +O\left(\frac{s_k^4}{\lambda^6}\right)=$$
 $$= 4\pi^2 +O\left(s_k^2\int_{x+\delta_{\lambda}(B_{Ar_k}(x_k)-x)}\left[\lambda^{-3}|x|^{-5} + \lambda^{-4}|x|^{-6}+\lambda^{-2}|x|^{-8}\right]\right) +O\left(\frac{s_k^4}{\lambda^6}\right)=$$
 $$= 4\pi^2 +O\left(s_k^2(\lambda r_k)^4 \left[\lambda^{-3}|\lambda r_k|^{-5} + \lambda^{-4}|\lambda r_k|^{-6}+\lambda^{-2}|\lambda r_k|^{-8}\right]\right) +O\left(\frac{s_k^4}{\lambda^6}\right)=$$
 $$= 4\pi^2 +O\left(\lambda^{-4}r_k^{-1}s_k^2 +\lambda^{-6}r_k^{-4}s_k^2\right) +O\left(\frac{s_k^4}{\lambda^6}\right).$$
\end{proof}

In a similar way, the following estimate can be deduced

\begin{lemma}\label{StimaCima}
 If $|x-x_k|\le R_k$ and $\lambda = \frac{\beta}{r_k}$, then
 $$\ci{J}_J(U_{x,\lambda}+v_{x,\lambda,s_k}) = 4\pi^2 + O(s_k^2r_k^2) +O\left(\frac{s_k^4}{\lambda^6}\right).$$
\end{lemma}

Thanks to these estimates, we can prove that $J$ is the required counterexample.

\begin{proof}[Proof of Theorem \ref{Teorema}]
 Choosing $J$ as described at the beginning of this section, choose $r_k=2^{-k}$, $R_k = C2^{-k}$ with $C$ large enough, $\alfa, \beta\gg 1$, $s_k= 2^{-2^k}$. Then, thanks to Lemmas \ref{StimaInterna}, \ref{StimaFondo}, \ref{StimaLaterale} and \ref{StimaCima}, $\max_{\Omega_k}\ci{J}_J(U_{x,\lambda}+v_{x,\lambda,s_k})>\max_{\de\Omega_k}\ci{J}_J(U_{x,\lambda}+v_{x,\lambda,s_k})$, and therefore there exists a critical point of $\ci{J}_J(U_{x,\lambda}+v_{x,\lambda,s_k})$ restricted to $\ci{M}$ in $\Omega_k$. By Proposition \ref{ProposizioneLS}, it is also a free critical point, and therefore a solution to Equation \eqref{EquazioneYamabeCR}.
 Finally since
 $$|B_{r_k}(x_k)|^{\frac{1}{4}}\max_{B_{r_k}(x_k)}(U_{x,\lambda}+v_{x,\lambda,s_k}) \ge$$
 $$\ge \N{(U_{x,\lambda}+v_{x,\lambda,s_k})}_{L^4(B_{r_k}(x_k))} \ge \N{U_{x,\lambda}}_{L^4(B_{r_k}(x_k))} - \N{v_{x,\lambda,s_k}}_X $$
 then $\max_{B_{r_k}(x_k)}(U_{x,\lambda}+v_{x,\lambda,s_k})\to\infty$.
\end{proof}

\appendix
\section{Estimates for the sublaplacian and Webster curvature}\label{Appendice}
On $\H^1$ consider the CR structure on $H(\H^1)$ generated by the vector field
$$\widetilde{Z} = Z+ f\con{Z}$$
where $|f|<1$ is a smooth complex function.

\begin{proposizione}\label{StimaCurvatura}
 There exist a constant $C$ such that
 $$\left|\widetilde{R} +\con{Z}^2f+Z\con{f} +\con{f}Z\con{Z}f +\con{f}\con{Z}Zf +fZ1\con{Z}\con{f}+f\con{Z}Z\con{f}+ \right.$$
 $$ \left.+|Zf|^2 +3|\con{Z}f|^2 \right| \le C\left( |f|^2|\nabla^2 f|+ |f||\nabla f|^2\right).$$
\end{proposizione}

\begin{proof}
 Define $\cerchio{\theta^1}$ in such a way that $(\cerchio{\theta^1},\cerchio{\theta^{\overline{1}}},\theta)$ is the dual frame to $(Z,\con{Z},T)$, and $\theta^1$ in such a way that $(\theta^1,\theta^{\overline{1}},\theta)$ is the dual frame to $(\widetilde{Z},\con{\widetilde{Z}},T)$.
 Then
 $$\theta^1 = \frac{1}{1-|f|^2}\cerchio{\theta^1} -\frac{\con{f}}{1-|f|^2}\cerchio{\theta^{\con{1}}}$$
 and
 $$\cerchio{\theta^1} = \theta^1 + \con{f}\theta^{\con{1}}.$$
 We recall (see Section 4 in \cite{L}) that the Tanaka-Webster connection is determined uniquely by the fact that the connection form $\omega_1^1$ and the torsion form $A_{\con{1}}^1$ satisfy the structure equations
 \begin{equation}\label{EquazioniDiStruttura}
  \left\{
   \begin{array}{ll}
    d\theta^1 = \theta^1\wedge\omega_1^1 + A_{\con{1}}^1\theta\wedge\theta^{\con{1}}\\
    \omega_1^1+\omega_{\con{1}}^{\con{1}} = dh_{1\con{1}}
   \end{array}\right.
 \end{equation}
 where $h_{1\con{1}}$ is the coefficient of Levi form, determined by $d\theta=ih_{1\con{1}}\theta^1\wedge\theta^{\con{1}}$.
 We need to determine $\omega_1^1$.
 $$d\theta^1 = \con{Z}\frac{1}{1-|f|^2}\cerchio{\theta^{\con{1}}}\wedge\cerchio{\theta^1} + T\frac{1}{1-|f|^2}\theta\wedge\cerchio{\theta^1} - Z\frac{\con{f}}{1-|f|^2}\cerchio{\theta^1}\wedge\cerchio{\theta^{\con{1}}} - T\frac{\con{f}}{1-|f|^2}\theta\wedge\cerchio{\theta^{\con{1}}} =$$
 $$= -\left(Z\frac{\con{f}}{1-|f|^2}+\con{Z}\frac{1}{1-|f|^2}\right)(\theta^1 + \con{f}\theta^{\con{1}})\wedge(f\theta^1 + \theta^{\con{1}}) + T\frac{1}{1-|f|^2}\theta\wedge(\theta^1 + \con{f}\theta^{\con{1}}) + $$
 $$ - T\frac{\con{f}}{1-|f|^2}\theta\wedge(\theta^{\con{1}}+f\theta^1) = $$
 $$ = -\left(Z\frac{\con{f}}{1-|f|^2}+\con{Z}\frac{1}{1-|f|^2}\right)(1-|f|^2)\theta^1\wedge\theta^{\con{1}}+$$
 $$+ \left( T\frac{1}{1-|f|^2} -fT\con{f}\frac{1}{1-|f|^2} -|f|^2T\frac{1}{1-|f|^2}\right)\theta\wedge\theta^1 +$$
 $$+\left(\con{f}T\frac{1}{1-|f|^2}-\con{f}T\frac{1}{1-|f|^2}-\frac{1}{1-|f|^2}T\con{f}\right)\theta\wedge\theta^{\con{1}}=$$
 $$ = -\left(Z\con{f}-\con{f}\frac{1}{1-|f|^2}Z(-|f|^2) -\frac{1}{1-|f|^2}\con{Z}(-|f|^2)\right)\theta^1\wedge\theta^{\con{1}}+$$
 $$+ \left( -(1-|f|^2)\frac{1}{(1-|f|^2)^2}T(-|f|^2) -fT\con{f}\frac{1}{1-|f|^2} \right)\theta\wedge\theta^1 -\frac{1}{1-|f|^2}T\con{f}\theta\wedge\theta^{\con{1}}=$$
 $$ = -\left(Z\con{f} +\frac{1}{1-|f|^2}|f|^2Z\con{f} + \frac{1}{1-|f|^2}\con{f}^2Zf +\frac{1}{1-|f|^2}\con{Z}(|f|^2)\right)\theta^1\wedge\theta^{\con{1}}+$$
 $$+ \frac{1}{1-|f|^2}\con{f}Tf\theta\wedge\theta^1 - \frac{1}{1-|f|^2}T\con{f}\theta\wedge\theta^{\con{1}}=$$
 $$= -\frac{1}{1-|f|^2}\left(Z\con{f} + \con{f}^2Zf +\con{Z}(|f|^2)\right)\theta^1\wedge\theta^{\con{1}}+$$
 \begin{equation}\label{ContoAppendice1}
  + \frac{1}{1-|f|^2}\con{f}Tf\theta\wedge\theta^1 - \frac{1}{1-|f|^2}T\con{f}\theta\wedge\theta^{\con{1}}.
 \end{equation}
 Since $h_{1\con{1}} = 2(1-|f|^2)$,
 $$dh_{1\con{1}} = -2d(|f|^2) = h_{1\con{1}}\omega_{\con{1}}^{\con{1}} + h_{1\con{1}}\omega_1^1,$$
 therefore
 $$\omega_1^1+\omega_{\con{1}}^{\con{1}} = -2\frac{1}{2(1-|f|^2)}\left(T(|f|^2)\theta + Z(|f|^2)\cerchio{\theta^1} + \con{Z}(|f|^2)\cerchio{\theta^{\con{1}}}\right) =$$
 \begin{equation}\label{ContoAppendice2}
  = -\frac{1}{(1-|f|^2)}\left(T(|f|^2)\theta + \left(Z(|f|^2) + f\con{Z}(|f|^2)\right)\theta^1 + \left(\con{Z}(|f|^2)+\con{f}Z(|f|^2)\right)\theta^{\con{1}}\right).
 \end{equation}
 From formulas \eqref{EquazioniDiStruttura}, \eqref{ContoAppendice1} and \eqref{ContoAppendice2} we deduce that
 $$\omega_1^1 = -\frac{1}{1-|f|^2}\con{f}Tf\theta + \con{Z}f\theta^1 -\frac{1}{1-|f|^2}\left(Z\con{f} + \con{f}^2Zf +\con{Z}(|f|^2)\right)\theta^{\con{1}}.$$
 Now, using formula 4.7 from \cite{L} we get that the Webster curvature is determined by
 $$d\omega_1^1 \operatorname{mod}\theta = R\theta^1\wedge\theta^{\con{1}}.$$
 Therefore we compute
 $$d\omega_1^1 \operatorname{mod}\theta = -\frac{1}{1-|f|^2}\con{f}Tf ih_{1\con{1}}\theta^1\wedge\theta^{\con{1}} + \con{\widetilde{Z}}\left(\con{Z}f\right)\theta^{\con{1}}\wedge\theta^1+$$
 $$-\con{Z}f\frac{1}{1-|f|^2}\left(Z\con{f} + \con{f}^2Zf +\con{Z}(|f|^2)\right)\theta^1\wedge\theta^{\con{1}}+$$
 $$-\widetilde{Z}\left(\frac{1}{1-|f|^2}\left(Z\con{f} + \con{f}^2Zf +\con{Z}(|f|^2)\right)\right) \theta^1\wedge\theta^{\con{1}}+$$
 $$- \frac{1}{1-|f|^2}\left(Z\con{f} + \con{f}^2Zf +\con{Z}(|f|^2)\right)\frac{1}{1-|f|^2}\left(\con{Z}f + f^2\con{Z}\con{f} +Z(|f|^2)\right)\theta^1\wedge\theta^{\con{1}} =$$
 $$= -2i\con{f}Tf\theta^1\wedge\theta^{\con{1}} - (\con{Z}+\con{f}Z)\left(\con{Z}f\right)\theta^1\wedge\theta^{\con{1}}+$$
 $$-\con{Z}f\frac{1}{1-|f|^2}\left(Z\con{f} + \con{f}^2Zf +\con{Z}(|f|^2)\right)\theta^1\wedge\theta^{\con{1}}+$$
 $$-(Z+f\con{Z})\left(\frac{1}{1-|f|^2}\left(Z\con{f} + \con{f}^2Zf +\con{Z}(|f|^2)\right)\right) \theta^1\wedge\theta^{\con{1}}+$$
 $$-\frac{1}{(1-|f|^2)^2}\left|Z\con{f} + \con{f}^2Zf +\con{Z}(|f|^2)\right|^2.$$
 Therefore
 $$R= -2i\con{f}Tf\theta^1\wedge\theta^{\con{1}} - (\con{Z}^2f+\con{f}Z\con{Z}f)+$$
 $$-\frac{1}{1-|f|^2}|\con{Z}f|^2 -\frac{1}{1-|f|^2}\con{f}^2Zf\con{Z}f -\frac{1}{1-|f|^2}\con{Z}f\con{Z}(|f|^2)+$$
 $$+\frac{1}{(1-|f|^2)^2}Z(-|f|^2)\left(Z\con{f} + \con{f}^2Zf +\con{Z}(|f|^2)\right)+$$
 $$-\frac{1}{1-|f|^2}\left(Z^2\con{f} + Z(\con{f}^2)Zf + \con{f}^2Z^2f +Z\con{Z}(|f|^2)\right) +$$
 $$+f\frac{1}{(1-|f|^2)^2}\con{Z}(-|f|^2)\left(Z\con{f} + \con{f}^2Zf +\con{Z}(|f|^2)\right)+$$
 $$-f\frac{1}{1-|f|^2}\left(\con{Z}Z\con{f} +\con{Z}(\con{f}^2)Zf + \con{f}^2\con{Z}Zf\right.+$$
 $$ \left.+\con{f}\con{Z}^2f + f\con{Z}^2\con{f}+2\con{Z}f\con{Z}\con{f}\right)+ $$
 $$-\frac{1}{(1-|f|^2)^2}\left|Z\con{f} + \con{f}^2Zf +\con{Z}(|f|^2)\right|^2 =$$
 $$= \con{f}(Z\con{Z}f-\con{Z}Zf) - \con{Z}^2f -\frac{1}{1-|f|^2}Z^2\con{f} -\frac{1}{1-|f|^2}\con{f}^2Z^2f +$$
 $$-f\frac{1}{1-|f|^2}\left(\con{f}\con{Z}^2f + f\con{Z}^2\con{f}\right)- \con{f}Z\con{Z}f+$$
 $$-\frac{1}{1-|f|^2}\left(fZ\con{Z}\con{f} + \con{f}Z\con{Z}f +Zf\con{Z}\con{f} +Z\con{f}\con{Z}f\right) +$$
 $$-\frac{f}{1-|f|^2}\con{Z}Z\con{f} -\frac{|f|^2}{1-|f|^2} \con{f}\con{Z}Zf+$$
 $$-\frac{1}{1-|f|^2}|\con{Z}f|^2 -\frac{\con{f}^2}{1-|f|^2}Zf\con{Z}f -\frac{\con{f}}{1-|f|^2}(\con{Z}f)^2 -\frac{f}{1-|f|^2}\con{Z}f\con{Z}\con{f}+$$
 $$-\frac{1}{(1-|f|^2)^2}(fZ\con{f} + \con{f}Zf)\left(Z\con{f} + \con{f}^2Zf \right)+$$
 $$-\frac{1}{(1-|f|^2)^2}\left|Z(|f|^2)\right|^2 -2\frac{1}{1-|f|^2}\left(\con{f}Z(\con{f})Zf\right) +$$
 $$-f\frac{1}{(1-|f|^2)^2}(\con{f}\con{Z}f+ f\con{Z}\con{f})\left(Z\con{f} + \con{f}^2Zf +\con{f}\con{Z}f+ f\con{Z}\con{f}\right)+$$
 $$-f\frac{1}{1-|f|^2}\left(2\con{f}|Zf|^2 + 2\con{Z}f\con{Z}\con{f}\right)+ $$
 $$-\frac{1}{(1-|f|^2)^2}\left|Z\con{f} + \con{f}^2Zf +\con{Z}(|f|^2)\right|^2 =$$
 $$= - \frac{1}{1-|f|^2}\con{Z}^2f -\frac{1}{1-|f|^2}Z^2\con{f} -\frac{1}{1-|f|^2}\con{f}^2Z^2f -f\frac{1}{1-|f|^2}\left(f\con{Z}^2\con{f}\right)+$$
 $$-\frac{\con{f}}{1-|f|^2}Z\con{Z}f -\frac{\con{f}}{1-|f|^2}\con{Z}Zf -\frac{f}{1-|f|^2}Z\con{Z}\con{f}-\frac{f}{1-|f|^2}\con{Z}Z\con{f}+$$
 $$-\frac{1}{1-|f|^2}\left(Zf\con{Z}\con{f} +Z\con{f}\con{Z}f\right) +$$
 $$-\frac{1}{1-|f|^2}|\con{Z}f|^2 -\frac{\con{f}^2}{1-|f|^2}Zf\con{Z}f -\frac{\con{f}}{1-|f|^2}(\con{Z}f)^2 -\frac{f}{1-|f|^2}\con{Z}f\con{Z}\con{f}+$$
 $$-\frac{1}{(1-|f|^2)^2}\left(f(Z\con{f})^2 + \con{f}|f|^2ZfZ\con{f} + \con{f}ZfZ\con{f} +\con{f}^3(Zf)^2\right)+$$
 $$-\frac{1}{(1-|f|^2)^2}\left|Z(|f|^2)\right|^2-2\frac{1}{1-|f|^2}\left(\con{f}Z(\con{f})Zf\right) +$$
 $$-\frac{|f|^2}{(1-|f|^2)^2}\left(|\con{Z}f|^2 + \con{f}^2Zf\con{Z}f +\con{f}(\con{Z}f)^2 + f\con{Z}f\con{Z}\con{f}\right)+$$
 $$-\frac{1}{(1-|f|^2)^2}\left(f^2Z\con{f}\con{Z}\con{f} + |f|^4|Zf|^2 +|f|^2f\con{Z}f\con{Z}\con{f}+ f^3(\con{Z}\con{f})^2\right)+$$
 $$-2\frac{|f|^2}{1-|f|^2}|Zf|^2 - 2\frac{1}{1-|f|^2}f\con{Z}f\con{Z}\con{f} + $$
 $$-\frac{1}{(1-|f|^2)^2}\left|Z\con{f} + \con{f}^2Zf +\con{Z}(|f|^2)\right|^2 =$$
 $$= - \frac{1}{1-|f|^2}\con{Z}^2f -\frac{1}{1-|f|^2}Z^2\con{f} -\frac{\con{f}^2}{1-|f|^2}Z^2f -\frac{f^2}{1-|f|^2}\con{Z}^2\con{f}+$$
 $$-\frac{\con{f}}{1-|f|^2}Z\con{Z}f -\frac{\con{f}}{1-|f|^2}\con{Z}Zf -\frac{f}{1-|f|^2}Z\con{Z}\con{f}-\frac{f}{1-|f|^2}\con{Z}Z\con{f}+$$
 $$-\frac{1}{1-|f|^2}\left(|Zf|^2 +|\con{Z}f|^2\right) +$$
 $$-\frac{1}{1-|f|^2}|\con{Z}f|^2 -\frac{\con{f}^2}{(1-|f|^2)^2}Zf\con{Z}f -\frac{\con{f}}{(1-|f|^2)^2}(\con{Z}f)^2 -\frac{3-|f|^2}{(1-|f|^2)^2}f\con{Z}f\con{Z}\con{f}+$$
 $$-\frac{f}{(1-|f|^2)^2}(Z\con{f})^2 -\frac{1}{(1-|f|^2)^2}\left(-|f|^2 + 3\right)\con{f}ZfZ\con{f} -\frac{1}{(1-|f|^2)^2}\con{f}^3(Zf)^2+$$
 $$-\frac{1}{(1-|f|^2)^2}\left|Z(|f|^2)\right|^2 +$$
 $$-\frac{f^2}{(1-|f|^2)^2}Z\con{f}\con{Z}\con{f} -\frac{2|f|^2 -|f|^4}{(1-|f|^2)^2}|Zf|^2 -\frac{f^3}{(1-|f|^2)^2}(\con{Z}\con{f})^2+$$
 $$-\frac{1}{(1-|f|^2)^2}\left|Z\con{f} + \con{f}^2Zf +\con{Z}(|f|^2)\right|^2.$$
 From this follows the thesis.
\end{proof}

\begin{proposizione}\label{DifferenzaSublaplaciani}
 There exist a constant $C$ such that
 $$\left|(\Delta_{\widetilde{J}}-\Delta_{J_0})u-\frac{1}{2}\left(f\con{Z}^2u+ \con{f}Z^2u\right)-\frac{3}{4}\left(Z\con{f}Zu +\con{Z}f\con{Z}u\right)\right| \lesssim$$
 $$\le C\left( |f|^2\Delta_{J_0}u +|f|^3|\nabla^2u| + |f|^2|\nabla f||\nabla u|\right).$$
\end{proposizione}

\begin{proof}
 Using the notation of the proof of the former Proposition,
 $$\Delta_bu = h^{1\con{1}}(\widetilde{Z}\con{\widetilde{Z}} + \con{\widetilde{Z}}\widetilde{Z} - \omega_1^1(\con{\widetilde{Z}})\widetilde{Z}u - \omega_{\con{1}}^{\con{1}}(\widetilde{Z})\con{\widetilde{Z}}u) =$$
 $$= h^{1\con{1}}(Z+f\con{Z})(\con{Z}+h^{1\con{1}}\con{f}Z)u + (\con{Z}+\con{f}Z)(Z+f\con{Z})u +$$
 $$+h^{1\con{1}}\frac{1}{1-|f|^2}\left(Z\con{f} + \con{f}^2Zf +\con{Z}(|f|^2)\right)(Z+f\con{Z})u +$$
 $$+ h^{1\con{1}}\frac{1}{1-|f|^2}\left(\con{Z}f + f^2\con{Z}\con{f} +Z(|f|^2)\right)(\con{Z}+\con{f}Z)u=$$
 $$= h^{1\con{1}}(Z\con{Z}u + f\con{Z}^2u + Z\con{f}Zu + |f|^2\con{Z}Zu) + $$
 $$+ h^{1\con{1}}(\con{Z}Zu + \con{f}Z^2u + \con{Z}f\con{Z}u + |f|^2Z\con{Z}u) + $$
 $$+h^{1\con{1}}\frac{1}{1-|f|^2}\left(Z\con{f} + \con{f}^2Zf +\con{Z}(|f|^2)\right)(Z+f\con{Z})u +$$
 $$+ h^{1\con{1}}\frac{1}{1-|f|^2}\left(\con{Z}f + f^2\con{Z}\con{f} +Z(|f|^2)\right)(\con{Z}+\con{f}Z)u=$$
 $$= \frac{1}{1-|f|^2}((1+|f|^2)\Delta_{J_0}u + f\con{Z}^2u + \con{f}Z^2u + Z\con{f}Zu + \con{Z}f\con{Z}u ) + $$
 $$+\frac{1}{2(1-|f|^2)^2}\left(Z\con{f} + 2\con{f}^2Zf +2\con{f}\con{Z}f +f\con{Z}\con{f}+f|f|^2\con{Z}\con{f} +|f|^2Z\con{f}\right)Zu +$$
 $$+ \frac{1}{2(1-|f|^2)^2}\left(\con{Z}f + 2f^2\con{Z}\con{f} +2fZ\con{f} +\con{f}Zf + \con{f}|f|^2Zf +|f|^2\con{Z}f\right)\con{Z}u.$$
 From this the thesis follows.
\end{proof}


\textsc{Claudio Afeltra, Department of Mathematics, University of Trento, Via Sommarive 14, 38123 Povo (Trento), Italy}
 
\textit{Email address}:  \texttt{claudio.afeltra@unitn.it}

\vspace{3mm}

\textsc{Andrea Pinamonti, Department of Mathematics, University of Trento, Via Sommarive 14, 38123 Povo (Trento), Italy}
 
\textit{Email address}:  \texttt{andrea.pinamonti@unitn.it}


\begin{thebibliography}{8}
 \bibitem[A]{A} Afeltra, Claudio;
 \textit{A compactness result for the CR Yamabe problem in three dimensions},
 Communications in Contemporary Mathematics.

 \bibitem[AM]{AM} Ambrosetti, Antonio; Malchiodi, Andrea;
 \textit{Perturbation methods and semilinear elliptic problems on $\R^n$},
 Progress in Mathematics, 240. Birkhäuser Verlag, Basel, 2006. xii+183 pp.

 \bibitem[B]{B} Brendle, Simon;
 \textit{Blow-up phenomena for the Yamabe equation},
 J. Amer. Math. Soc. 21 (2008), no. 4, 951–979.
 
 \bibitem[BM]{BM} Brendle, Simon; Marques, Fernando C.;
 \textit{Blow-up phenomena for the Yamabe equation. II},
 J. Differential Geom. 81 (2009), no. 2, 225–250.
 
 \bibitem[CGS]{CGS} Caffarelli, Luis A.; Gidas, Basilis; Spruck, Joel;
 \textit{Asymptotic symmetry and local behavior of semilinear elliptic equations with critical Sobolev growth},
 Comm. Pure Appl. Math. 42 (1989), no. 3, 271–297.
 
 \bibitem[CLMR]{CLMR} Giovanni Catino; Yanyan Li; Dario D. Monticelli; Alberto Roncoroni;
 \textit{A Liouville theorem in the Heisenberg group},
 arXiv:2310.10469 (2023).
 
 
 \bibitem[CMY]{CMY} Cheng, Jih-Hsin; Malchiodi, Andrea; Yang, Paul;
 \textit{On the Sobolev quotient of three-dimensional CR manifolds},
 Rev. Mat. Iberoam. 39 (2023), no. 6, 2017–2066.
  
 \bibitem[DT]{DT} Dragomir, Sorin; Tomassini, Giuseppe;
 \textit{Differential geometry and analysis on CR manifolds}, Progress in Mathematics, 246. Birkhäuser Boston, Inc., Boston, MA, 2006. xvi+487 pp.
 
 \bibitem[G]{G} Gamara, Najoua;
 \textit{The CR Yamabe conjecture—the case n=1},
 J. Eur. Math. Soc. (JEMS) 3 (2001), no. 2, 105–137.
 
 \bibitem[GY]{GY} Gamara, Najoua; Yacoub, Ridha;
 \textit{CR Yamabe conjecture—the conformally flat case},
 Pacific J. Math. 201 (2001), no. 1, 121–175.
 
 \bibitem[JL1]{JL1} Jerison, David; Lee, John M.;
 \textit{The Yamabe problem on CR manifolds},
 J. Differential Geom. 25 (1987), no. 2, 167–197.
 
 \bibitem[JL2]{JL2}Jerison, David; Lee, John M.;
 \textit{Extremals for the Sobolev inequality on the Heisenberg group and the CR Yamabe problem},
 J. Amer. Math. Soc. 1 (1988), no. 1, 1–13.
 
 \bibitem[JL3]{JL3} Jerison, David; Lee, John M.;
 \textit{Intrinsic CR normal coordinates and the CR Yamabe problem},
 J. Differential Geom. 29 (1989), no. 2, 303–343.
 
 \bibitem[KMS]{KMS} Khuri, M. A.; Marques, F. C.; Schoen, R. M.;
 \textit{A compactness theorem for the Yamabe problem},
 J. Differential Geom. 81 (2009), no. 1, 143–196.
 
 \bibitem[L]{L} Lee, John M.,
 \textit{The Fefferman metric and pseudo-Hermitian invariants},
 Trans. Amer. Math. Soc. 296 (1986), no. 1, 411–429.
 
 \bibitem[LP]{LP} Lee, John M.; Parker, Thomas H.;
 \textit{The Yamabe problem},
 Bull. Amer. Math. Soc. (N.S.) 17 (1987), no. 1, 37–91.

 \bibitem[MU]{MU} Malchiodi, Andrea; Uguzzoni, Francesco,
 \textit{A perturbation result for the Webster scalar curvature problem on the CR sphere},
 J. Math. Pures Appl. (9) 81 (2002), no. 10, 983–997.
\end{thebibliography}
\end{document}